\documentclass[12pt]{article}

\usepackage[margin = 1.2in]{geometry}
\usepackage{xcolor}

\usepackage{amsmath,amsthm,amssymb,amsfonts}
\usepackage{mathtools}
\usepackage{bm} 
\usepackage{MnSymbol}
\usepackage[numbers,square,sort]{natbib}

\usepackage[colorlinks,linkcolor=red,anchorcolor=blue,citecolor=green]{hyperref}
\usepackage[capitalise,noabbrev]{cleveref}

\newtheorem{theorem}{Theorem}

\newtheorem{corollary}[theorem]{Corollary}
\newtheorem{definition}[theorem]{Definition}
\newtheorem{example}[theorem]{Example}
\newtheorem{lemma}[theorem]{Lemma}

\newtheorem{remark}[theorem]{Remark}

\newcommand{\CC}{\mathbb{C}}
\newcommand{\QQ}{\mathbb{Q}}
\newcommand{\ZZ}{\mathbb{Z}}
\newcommand{\FF}{\mathbb{F}}

\DeclareMathOperator{\fwe}{fwe}

\DeclareMathOperator{\ccwe}{ccwe}
\DeclareMathOperator{\wt}{wt}
\DeclareMathOperator{\Span}{span}

\DeclareMathOperator{\GL}{GL}
\DeclareMathOperator{\AGL}{AGL}

\DeclareMathOperator{\Hom}{Hom}

\DeclareMathOperator{\FS}{FS}
\DeclareMathOperator{\Inv}{Inv}

\DeclarePairedDelimiter{\set}{\{}{\}}

\DeclarePairedDelimiter{\abs}{\vert}{\vert}

\DeclarePairedDelimiter{\lrangle}{\langle}{\rangle}
\DeclarePairedDelimiter{\lrparath}{(}{)}

\begin{document}

\title{The complex conjugate invariants of Clifford groups}

\author{Eiichi Bannai \footnote{Professor Emeritus of Kyushu University, Fukuoka, Japan.
Postal Address: Asagaya-minami 3-2-33, Suginami-ku, Tokyo 166-0004, Japan. \href{bannai@math.kyushu-u.ac.jp}{bannai@math.kyushu-u.ac.jp}}
	\and
        Manabu Oura \footnote{Institute of Science and Engineering, Kanazawa University, Kakuma-machi, Kanazawa, Ishikawa 920-1192, Japan. \href{oura@se.kanazawa-u.ac.jp}{oura@se.kanazawa-u.ac.jp}}
    \and
        Da Zhao \footnote{School of Mathematical Sciences, Shanghai Jiao Tong University, 800 Dongchuan Road, Minhang District, Shanghai 200240, China. \href{jasonzd@sjtu.edu.cn}{jasonzd@sjtu.edu.cn}}
        }

\date{}

\maketitle

\begin{abstract}

Nebe, Rains and Sloane studied the polynomial invariants for real and complex Clifford groups and they relate the invariants to the space of complete weight enumerators of certain self-dual codes. 
The purpose of this paper is to show that very similar results can be obtained for the invariants of the complex Clifford group $\mathcal{X}_m$ acting on the space of conjugate polynomials in $2^m$ variables of degree $N_1$ in $x_f$ and of degree $N_2$ in their complex conjugates $\overline{x_f}$. 
In particular, we show that the dimension of this space is $2$, for $(N_1,N_2)=(5,5)$.
This solves the Conjecture 2 given in Zhu, Kueng, Grassl and Gross affirmatively. 
In other words if an orbit of the complex Clifford group is a projective $4$-design, then it is automatically a projective $5$-design.



\end{abstract}

\section{Motivation and Background}

The original motivation of this paper was to settle Conjecture 2 in page 26 of Zhu-Kueng-Grassl-Gross \cite{Zhu2016Sep}: The Clifford group fails gracefully to be a unitary 4-design, arXiv: 1609.08172. 
The Conjecture 2 there says if an orbit of the complex Clifford group is a projective $4$-design, then it is automatically a projective $5$-design. 
This is equivalent to the statement that $a_{4,4} = a_{5,5}=2$ in \cref{empl:Erinaceus} in this paper. 
So the validity of Conjecture 2 was proved.
The proof goes parallel with Nebe-Rains-Sloane's proof in \cite{MR1845897}. 

The aim of design theory is to approximate a space by its finite subset. 
There have been numerous study of designs on intervals, spheres, and $\binom{X}{k}$, namely the $k$-subsets of a $v$-element set $X$ \cite{MR3594369,MR2246267}. 
These designs are useful in areas such as numerical computation and experiment design. 
Experiments and engineering related to quantum physics raise the need for approximating the space of unitary groups and complex spheres. 
A unitary $t$-design is a subset $X$ of the unitary group $U(d)$ such that the averaging of every function $f \in \Hom_{(t,t)}(U(d))$ over $X$ is equal to that over $U(d)$. 
Here $\Hom_{(t,t)}(U(d))$ is the space of homogeneous complex conjugate polynomials in the entries of the unitary matrix as well as their complex conjugates that are both of degree $t$. 
Similarly a projective $t$-design is a subset of the complex sphere with functions ranging from $\Hom_{t,t}(\CC^d)$, homogeneous complex conjugate polynomials on complex sphere. 
Several such designs are known \cite{MR4051739,BNOZ,bannai2020,MR2529619}. 
The complex Clifford group, which is relatively easy to implement in quantum physics, outstands as an infinite family of unitary $3$-designs \cite{Webb:2016:CGF:3179439.3179447}. 
It is pointed out in that the complex Clifford group fails gracefully to be a unitary $4$-design, and projective $5$-designs may come for free from projective $4$-designs by the complex Clifford group \cite{Zhu2016Sep}. 

The invariants of (complex) Clifford group of genus $m$ were characterized by Runge's theorem \cite{MR1207281,MR1339948,MR1368288}. 
It says that the space of polynomial invariants is spanned by the genus-$m$ complete weight enumerators of binary (doubly) even self-dual codes. 
This relates the Clifford group to coding theory, which is another prosperous area since Shannon introduces the concept of information entropy \cite{MR26286}. 
Indeed Runge obtained the above result through the study of Siegel modular form. 
The automorphism group of the Barnes-Wall lattice is an index $2$ subgroup of the real Clifford group \cite{MR0142666,MR0125874,MR148736}. 
Nebe-Rains-Sloane tried to establish the parallel theory between self-dual codes and unimodular lattices, and they propose the Weight Enumerator Conjecture for finite form ring in their book \cite{MR2209183}. 
The result in the present paper can basically be regarded as giving another special case of the Weight Enumerator Conjecture.
\section{Complex Clifford group}

We basically follow the definitions and notation by Nebe-Rains-Sloane in \cite{MR1845897} and \cite{MR2209183}.

\begin{definition}
	The complex Clifford group of genus $m$, denoted by $\mathcal{X}_m$, is generated by the following elements of $U(2^m)$. 
	Let $e_v$, $v \in \mathbb{F}_2^m$ be an orthonormal basis of $\mathbb{C}^{2^m}$. 
	\begin{enumerate}
		\item Diagonal elements $d_S$ defined by $d_S(e_v) = i^{S[v]} e_v$ for every symmetric integral matrix $S$ of size $m \times m$ and every $v \in \mathbb{F}_2^m$. 
		Here $S[v] = v^T S v$ where the entries of $v$ are regarded as integers.
		\item Permutation elements $m_{g,v'}$ defined by $m_{g,v'}(e_v) = e_{gv+v'}$ for every $(g,v') \in \GL(m,2) \ltimes \FF_2^m= \AGL(m,2)$ and every $v \in \mathbb{F}_2^m$. 
		\item Hadamard element $h \otimes I_2 \otimes \cdots \otimes I_2 \in U(2)^{\otimes m}$, where $h = \frac{1}{\sqrt{2}} \begin{bmatrix}
			1 & 1 \\
			1 & -1
		\end{bmatrix}$.
	\end{enumerate}
\end{definition}

\begin{remark}
	Let $\sigma_x = \begin{bmatrix}
		0 & 1 \\
		1 & 0
	\end{bmatrix}$, $\sigma_y = \begin{bmatrix}
		0 & -i \\
		i & 0
	\end{bmatrix}$ and $\sigma_z = \begin{bmatrix}
		1 & 0 \\
		0 & -1
	\end{bmatrix}$ be the Pauli operators. 
	Let $\phi = \begin{bmatrix}
		1 & 0 \\
		0 & i
	\end{bmatrix}$ be the $\pi/2$-phase gate.
	We denote by $E_{i,j}$ the matrix unit whose $(i,j)$-entry is $1$. 
	Then another set of generators of $\mathcal{X}_m$ is
	$\sigma_x \otimes I_2 \otimes \cdots \otimes I_2$, 
	$h \otimes I_2 \otimes \cdots \otimes I_2$, 
	$\phi \otimes I_2 \otimes \cdots \otimes I_2$,
	$d_{S_{1,2}}$ and $\GL(m,2)$, where $S_{1,2} = E_{1,2} + E_{2,1}$. 
	Note that the generator 
	$\phi \otimes I_2 \otimes \cdots \otimes I_2$ is exactly $d_{E_{1,1}}$.
\end{remark}

\section{Weight enumerators}

\begin{definition}
	Let $q$ a be a prime power and $N_1, N_2$ non-negative integers.
	A \emph{linear code} $C$ of length $(N_1,N_2)$ over $V=\FF_q$, or a code for simplicity, is a subspace of $V^{N_1+N_2}$. 
	Each element $c$ of a code $C$ is called a \emph{codeword}. 
\end{definition}

In particular, a code over $\FF_2$ is called a binary code. 
We focus mostly on codes over $\FF_2$ or $\FF_{2^m}$.

\begin{definition} \label{def:clairvoyance}
	Let $C$ be a linear code over $\FF_q$ of length $(N_1,N_2)$. 
	For two codewords $x = (x_1, \ldots, x_{N_1 + N_2})$ and $y = (y_1, \ldots, y_{N_1 + N_2})$, the bilinear form $\beta(x,y)$ is defined by $\beta(x,y) = x_1 y_1 + \cdots + x_{N_1} y_{N_1} - x_{N_1+1} y_{N_1+1} - \cdots - x_{N_1 + N_2} y_{N_1+N_2} \in \FF_q$. 
	We call $C$ \emph{self-orthogonal} if $C \subseteq C^\perp$. 
	And we call $C$ \emph{self-dual} if $C = C^\perp$. 
	Here $$C^\perp = \set{c' \in \FF_q^{N_1+N_2} : \beta(c,c') = 0, \forall c \in C}.$$
\end{definition}

\begin{remark}
	This bilinear form coincides with the inner product when the field is  $\mathbb{F}_2$. 
\end{remark}

\begin{definition}
	Let $C$ be a binary linear code of length $(N_1,N_2)$. 
	For each codeword $c = (c_1, \ldots, c_{N_1 + N_2})$, we define its weight by $wt(c) = c_1 + \cdots + c_{N_1} - c_{N_1+1} - \cdots - c_{N_1 + N_2} \in \mathbb{Z}$. 
	We call $C$ \emph{even} if the weight of each codeword is an even number. 
	And we call $C$ \emph{doubly-even} if the weight of each codeword is divisible by $4$.
\end{definition}

\begin{remark}
	Usually we think of the weight of a codeword as an integer. 
	Again we regard elements of $\FF_p$ as elements of $\ZZ$. 
	Note that the classical binary linear code corresponds to the binary linear code here with $N_2 = 0$. 
\end{remark}

\begin{definition}
	The \emph{full weight enumerator} of a code $C \leq V^{N_1+N_2}$ is the element 
	\[
		\fwe(C) := \sum_{c \in C} e_c
	\]
	in the group algebra $\CC[V^{N_1+N_2}]$ where $e_c$ is regarded as a symbol.
\end{definition}

\begin{definition}
	The \emph{genus-$m$ full weight enumerator} of a code $C \leq V^{N_1+N_2}$ is the element 
	\[
		\fwe_m(C) := \sum_{c^1, \ldots, c^m \in C} e_{c^1, \ldots, c^m} 
	\]
	in the group algebra $\CC[V^{mN_1+mN_2}] \cong \otimes^m \CC[V^{N_1+N_2}]$.
\end{definition}

\begin{definition} \label{def:tarlike}
	The \emph{complete conjugate weight enumerator} $\ccwe(C)$ of a code $C \leq V^{N_1+N_2}$ is the projection under $\pi$ of the full weight enumerator of $C$ to the space $\CC[x_f, \bar{x}_f | f \in V]$, where $\pi$ is the mapping defined by $e_c \mapsto x_{c_1} \cdots x_{c_{N_1}} \overline{x_{c_{N_1+1}}} \cdots \overline{x_{c_{N_1+N_2}}}$ for $c = (c_1, \ldots, c_{N_1+N_2})$. 
	In other words
	\[ \ccwe(C) = \pi(\fwe(C)).\]
\end{definition}

In particular, for a binary code $C$ of length $(N_1,N_2)$ the complete conjugate weight enumerator of $C(m)$ is a homogeneous polynomial in $2^m$ variables of degree $N_1$ in $x_f$ and of degree $N_2$ in their complex conjugates. 
Such kind of polynomial are called multivariate conjugate complex polynomial (sometimes abbreviated as conjugate polynomial).

\begin{remark} \label{rmk:synapticulum}
	Let $C \leq V^{N_1+N_2}$ be a linear code. 
	For $m \in \mathbb{N}$, let $C(m) := C \otimes_{\mathbb{F}_q} \mathbb{F}_{q^m}$ be the extension of $C$ to a code over the field $\mathbb{F}_{q^m}$. 
	The isomorphism $V^{N_1+N_2} \otimes_{\FF_q} \FF_{q^m} \cong V^{mN_1+mN_2}$ gives us $\fwe_m(C) = \fwe(C(m))$. 
	The element in $C(m)$ can be represented as an $m \times (N_1 + N_2)$ matrix $M$ with the rows being the elements of $C$. 
	The first $N_1$ columns of $M$ contribute to the $2^m$ variables and the last $N_2$ columns contribute to their complex conjugates.
\end{remark}

\begin{theorem}[{Generalized MacWilliams identity \cite[Example 2.2.6]{MR2209183}}]
	Let $\beta \in Bil(V, \QQ/\ZZ)$ be a nonsingular bilinear form. 
	Then for any code $C \leq V$, the full weight enumerator of $C^\perp = C^{\perp,\beta}$ is given by 
	\[
		\fwe(C^\perp) = \frac{1}{\abs{C}} \sum_{w \in V} \sum_{v \in C} \exp(2\pi i \beta(w,v)) e_w.
	\]
	In other words, the full weight enumerator of $C^\perp$ can be obtained by changing $e_v$ to $\sum_{v \in C} \exp(2\pi i \beta(w,v)) e_w$ in the full weight enumerator of $C$, divided by $\abs{C}$
\end{theorem}

\begin{theorem}[Analogue of {\cite[Theorem 3.5]{MR1845897}}] \label{thm:misshape}
	Let $C$ be a binary doubly-even self-dual code of length $(N_1,N_2)$. 
	\begin{enumerate}
		\item \label{itm:Iranic} The complex Clifford group $\mathcal{X}_m$ fixes the full weight enumerator $\fwe(C(m))$. 
		\item The complex Clifford group $\mathcal{X}_m$ fixes the complete conjugate weight enumerator $\ccwe(C(m))$. 
	\end{enumerate}
\end{theorem}

\begin{proof}
	Note that every codeword is orthogonal to itself, hence orthogonal to the all-one codeword.
	Since $C$ is doubly-even and self-dual, the all-one codeword in contained in $C$. 
	Therefore the difference $N_1 - N_2$ is necessarily a multiple of $4$. 
	The complex Clifford group $\mathcal{X}_m$ acts on $\CC[V^{N_1+N_2}] \cong \otimes^{N_1+N_2} (\CC^{2^m})$ diagonally. 
	Keep in mind that the action at the last $N_2$ entries comes with a complex conjugation. 
	Since the action commutes with the projection $\pi$, by \cref{rmk:synapticulum} we only need to prove the first statement. 
	It is enough to consider the generators of $\mathcal{X}_m$.

	The generators 
	$\sigma_x \otimes I_2 \otimes \cdots \otimes I_2$, 
	$h \otimes I_2 \otimes \cdots \otimes I_2$,
	$\phi \otimes I_2 \otimes \cdots \otimes I_2$
	are of the form 
	$u \otimes I_2 \otimes \cdots \otimes I_2$. 
	It suffices to consider $m=1$ for these generators. 
	The matrix $\sigma_x$ acts as $(\otimes^{N_1} \sigma_x) \otimes (\otimes^{N_2} \overline{\sigma_x})$ on $\otimes^{N_1+N_2} (\CC^2)$. 
	It maps a codeword $c = (c_1, \ldots, c_{N_1+N_2})$ to $c + \bm{1}$, where $\bm{1}$ is the all-one vector. 
	Since $C$ is self-dual, thus $\bm{1} \in C$. 
	So $\sigma_x$ permutes the codewords and hence preservers the full weight enumerator. 
	The matrix $\phi$ maps $e_c$ to $i^{\wt(c)}e_c$, which is equal to $e_c$ since $C$ is doubly-even. 
	Finally the MacWilliams identity for full weight enumerators guarantees the matrix $h$ preservers $\fwe(C(m))$. 

	The generator $d_{S_{1,2}}$ only occurs for $m \geq 2$. 
	Still we only need to consider the case $m=2$ for this generator. 
	For a pair of codewords $c = (c_1, \ldots, c_{N_1+N_2})$ and $c' = (c_1', \ldots, c_{N_1+N_2}')$, the matrix $d_{S_{1,2}}$ maps $e_{c,c'}$ to $i^{2\beta(c,c')} e_{c,c'}$, which is equal to $e_{c,c'}$ since $C$ is self-dual. 
	So the generator $d_{S_{1,2}}$ preservers $\fwe(C(m))$ as well. 

	Finally the generators $\GL(m,2)$ permute the elements of $\FF_2^m$. 
	Let $M$ be the $m \times (N_1 + N_2)$ matrix represent a codeword in $C(m)$ and let $g$ be an element of $\GL(m,2)$. 
	Then $m_g(e_M) = e_{gM}$ which is still an element in $C(m)$. 
	Since $g$ is invertible, the matrix $m_g$ only permutes the codewords. 
	Hence these generators fix the full weight enumerators of $C(m)$.
\end{proof}

\section{\texorpdfstring{The ring of conjugate invariants of $\mathcal{X}_m$}{The ring of conjugate invariants of X\_m}}  \label{sec:demibombard}

\begin{definition}
	A conjugate polynomial $p$ in $2^m$ variables is called a \emph{complex Clifford invariant of genus $m$} if it is invariant under the complex Clifford group $\mathcal{X}_m$. 
	In particular, it is called a \emph{parabolic invariant} if it is invariant under the parabolic subgroup $P$ generated by the diagonal elements and permutation elements of $\mathcal{X}_m$, and a \emph{diagonal invariant} if it is invariant under the diagonal elements of $\mathcal{X}_m$.
\end{definition}

\begin{lemma}[Analogue of {\cite[Lemma 4.2]{MR1845897}}]
	A conjugate polynomial $p$ is a diagonal invariant if and only if all of its monomials are diagonal invariants. 
\end{lemma}

Let $M$ be an $m \times (N_1 + N_2)$ matrix over $\FF_2$. 
We can associate it to a monic conjugate monomial $\nu_M \in \CC[x_f, \overline{x_f} : f \in \FF_{2^m}]$ by taking the product of variables corresponding to its columns with complex conjugate at the last $N_2$ columns. 

\begin{lemma}[Analogue of {\cite[Theorem 4.3]{MR1845897}}]
	A monoic conjugate monomial $\nu_M$ is a diagonal invariant if and only if the rows of $M$ are orthogonal and doubly-even as codewords. 
\end{lemma}

\begin{proof}
	It suffices to consider the action of elements $d_{E_{k,k}}$ and $d_{E_{k,l}+E_{l,k}}$ for $1 \leq k \neq l \leq m$. 
	The effect of the action of $d_{S_{k,k}}$ is to multiply $\nu_M$ by $i^{\wt(c^k)}$ where $\wt(c^k)$ is the weight of the $k$-th row of $M$. 
	This demands that each row of $M$ is doubly-even.
	The effect of the action of $d_{S_{k,l}}$ is to multiply $\nu_M$ by $i^{2\beta(c^k,c^l)}$. 
	This demands the rows of $M$ to be orthogonal.
\end{proof}

For $(g,v') \in \GL(m,2) \ltimes \FF_2^m = A\GL(m,2)$, its action on $\nu_M$ is given by $m_{g,v'}(\nu_M) = \nu_{g M + v'\bm{1}^T}$. 
Then the orbit of $\nu_M$ under $A\GL(m,2)$ is a code $C$ containing $\bm{1}^T$ and of dimension at most $m+1$. 
This allows us to define a conjugate polynomial $\nu_m(C)$ for any binary code $C$. 
\[
	\nu_m(C) := 
		\sum_{\substack{M \in \FF_2^{m \times (N_1 + N_2)} \\ \Span\lrangle{M,\bm{1}^T} = C}} \nu_M
\]
Note that $\nu_m(C) = 0$ if $\bm{1}^T \notin C$ or $\dim(C) > m+1$.
In particular, if $C$ is doubly-even self-orthogonal, then $\nu_m(C)$ is a parabolic invariant.

\begin{lemma}[Analogue of {\cite[Theorem 4.4]{MR1845897}}]
	A basis of the space of parabolic invariants of degree $(N_1,N_2)$ is given by conjugate polynomials of the form $\nu_m(C)$ where $C$ ranges over equivalence classes of binary self-orthogonal codes of length $(N_1, N_2)$ containing $\bm{1}^T$ and of dimension at most $m+1$.
\end{lemma}

\begin{lemma}[Analogue of {\cite[Theorem 4.5]{MR1845897}}] \label{lem:Alebion}
	For any binary code $C$,
	\[
		\ccwe(C(m)) = \sum_{D \subseteq C} \nu_m(D).
	\]
\end{lemma}

\begin{proof}
	By definition,
	\[
		\ccwe(C(m)) = \sum_M \nu_M
	\]
	where $M$ ranges over $m \times (N_1 + N_2)$ with all rows in $C$. 
	Let $M$ be such a matrix. 
	Then $M$ uniquely determines a subcode $D=\Span\lrangle{M,\bm{1}^T}$ of $C$. 
	Therefore 
	\[
		\ccwe(C(m)) = \sum_{D \subseteq C} \sum_{\Span\lrangle{M,\bm{1}^T}=D} \nu_M = \sum_{D \subseteq C} \nu_m(D). \qedhere
	\]
\end{proof}

Indeed the complete conjugate weight enumerator of $C(m)$ is essentially the partial sum of the function $\nu_m$ in the subspace poset of binary codes. 
So we have the Möbius inversion formula. 

\begin{corollary} \label{coro:striker}
	For any binary code $D$,
	\[
		\nu_m(D) = \sum_{C \subseteq D} \ccwe(C(m)) \mu(C,D)
	\]
	where $\mu(C,D)$ is the Möbius function of the subspace poset of binary codes.
\end{corollary}

\begin{theorem}[Analogue of {\cite[Theorem 4.6]{MR1845897}}] \label{thm:gluttonish}
	The space of parabolic invariants is spanned by the conjugate polynomials $\ccwe(C(m))$, where $C$ ranges over equivalence classes of binary self-orthogonal codes containing $\bm{1}^T$ and of dimension at most $m+1$.
\end{theorem}

\begin{proof}
	Note that a subcode of a doubly-even self-orthogonal code is also doubly-even self-orthogonal and $\nu_m(D) = 0$ if $\bm{1}^T \notin D$. 
	The theorem follows from \cref{lem:Alebion,coro:striker}.
\end{proof}

Let $X_P$ be the operation of averaging over the parabolic subgroup $P$, namely $X_P = \frac{1}{\abs{P}} \sum_{g \in P} g$. 

\begin{lemma}[Analogue of {\cite[Lemma 4.7]{MR1845897}}] \label{lem:infinitesimally}
	For any binary doubly-even self-orthogonal code $C$ of length $(N_1,N_2)$ containing $\bm{1}$ and of dimension $(N_1+N_2)/2−r$,
	\begin{align*}
		&X_P (h \otimes I_2 \otimes \cdots \otimes I_2) \ccwe(C(m)) \\
		=& \frac{2^{m-r} - 2^r}{2^m - 1} \ccwe(C(m)) + \frac{2^{-r}}{2^m - 1} \sum_{\substack{C \subsetneq C' \subseteq C'^\perp  \\ [C' : C] = 2}} \ccwe(C'(m)),
	\end{align*}
	where the last summation is over all doubly-even self-orthogonal codes containing $C'$ containing $C$ to index $2$.
\end{lemma}

\begin{proof}
	By the MacWilliams identity, we have 
	\[
		(h \otimes I_2 \otimes \cdots \otimes I_2) \ccwe(C(m)) = 2^{-r} \sum \nu_M,
	\]
	where $M$ ranges over $m \times (N_1+N_2)$ matrices such that the first row of $M$ is in $C^\perp$ and the remaining rows are in $C$. 0
	For each code $\bm{1} \in D \subseteq C^\perp$, consider the partial sum over the matrices with $\lrangle{M,\bm{1}} = D$. 
	If $D \subseteq C$, then the partial sum is indeed $\nu_m(D)$, which is invariant under $X_P$. 
	Otherwise we have $[D: D \cap C] = 2$. 
	Since some elements do not belong to $C$, we use an vector $\chi_M \in \FF_2^m$ to indicate whether or not a row of $M$ is in $C$. 
	In particular, $(\chi_M)_i = 1$ if the $i$-th row of $M$ is in $C^\perp \backslash C$, and $(\chi_M)_i = 0$ otherwise. 
	Now we consider the partial sum 
	\[
		\sum_{\substack{\lrangle{M,\bm{1}} = D \\ \chi_M = (1,0, \ldots, 0)}} \nu_M.
	\] 
	If $D$ is not isotropic, then the partial sum is annihilated by the diagonal subgroup of $\mathfrak{X}_m$. 
	On the other hand, if $D$ is isotropic, then the effect of apply an element of $\AGL(m,2)$ to this sum is inequivalent to changing $\chi_M$. 
	So for $D \subseteq D^\perp$, we have 
	\[
		X_P \sum_{\substack{\lrangle{M,\bm{1}} = D \\ \chi_M = (1,0, \ldots, 0)}} \nu_M = \frac{1}{\abs{\chi \in \FF_2^m : \chi \neq 0}} \nu_m(D). 
	\]
	Therefore 
	\[
		X_P (h \otimes I_2 \otimes \cdots \otimes I_2) \ccwe(C(m)) = 2^{-r} \sum_{\bm{1} \in D \subseteq C} \nu_m(D) + \frac{2^{-r}}{2^m - 1} \sum_{\substack{\bm{1} \in D \subseteq C^\perp \subseteq C^\perp \\ [D : D \cap C] = 2}} \nu_m(D).
	\]
	We denote by $C'$ the code generated by $\lrangle{D,C}$. 
	Then we have $D \subseteq C^\perp$, $C' \subseteq C'^\perp$ if and only if $D \subseteq D^\perp$. 
	We can rewrite the summation as 
	\[
		X_P (h \otimes I_2 \otimes \cdots \otimes I_2) \ccwe(C(m)) = 2^{-r} \sum_{\bm{1} \in D \subseteq C} \nu_m(D) + \frac{2^{-r}}{2^m - 1} \sum_{\substack{C \subsetneq C' \subseteq C'^\perp  \\ [C' : C] = 2}} \sum_{\substack{\bm{1} \in D \subseteq C' \\ D \not\subseteq C}}\nu_m(D).
	\]
	Since every code $C' = \lrangle{D,C}$ contains each subcode of $C$ exactly once, so we can split the summation of $D$ as follows. 
	\begin{align*}
		& X_P (h \otimes I_2 \otimes \cdots \otimes I_2) \ccwe(C(m)) \\
		= & 2^{-r} \sum_{\bm{1} \in D \subseteq C} \nu_m(D) + \frac{2^{-r}}{2^m - 1} \sum_{\substack{C \subsetneq C' \subseteq C'^\perp  \\ [C' : C] = 2}} \sum_{\substack{\bm{1} \in D \subseteq C' \\ D \not\subseteq C}}\nu_m(D) \\
		= & 2^{-r} \sum_{\bm{1} \in D \subseteq C} \nu_m(D) + \frac{2^{-r}}{2^m - 1} \sum_{\substack{C \subsetneq C' \subseteq C'^\perp  \\ [C' : C] = 2}} \lrparath*{\sum_{\bm{1} \in \substack{D \subseteq C'}}\nu_m(D) - \sum_{\bm{1} \in D \subseteq C}\nu_m(D)} \\
		= & 2^{-r} \sum_{\bm{1} \in D \subseteq C} \nu_m(D) + \frac{2^{-r}}{2^m - 1} \sum_{\substack{C \subsetneq C' \subseteq C'^\perp  \\ [C' : C] = 2}} \sum_{\bm{1} \in \substack{D \subseteq C'}}\nu_m(D) - \frac{2^{-r}}{2^m - 1} (2^{2r}-1) \sum_{\bm{1} \in D \subseteq C}\nu_m(D) \\
		= & \frac{2^{m-r} - 2^r}{2^m - 1} \ccwe(C(m)) + \frac{2^{-r}}{2^m - 1} \sum_{\substack{C \subsetneq C' \subseteq C'^\perp  \\ [C' : C] = 2}} \ccwe(C'(m)).
	\end{align*}
\end{proof}

\begin{lemma}[{\cite[Lemma 4.8]{MR1845897} or \cite[Lemma 5.5.10]{MR2209183}}] \label{lem:glossographical}
	Let $V$ be a finite dimensional vector space, $M$ a linear transformation on $V$, and $P$ a partially ordered set. 
	Suppose there exists a spanning set $v_p$ of $V$ indexed by $p \in P$ on which $M$ acts triangularly, that is,
	\[
		M v_p = \sum_{q \geq p} c_{pq} v_q,
	\]
	for suitable coefficients $c_{pq}$. 
	Furthermore that $c_{pp} = 1$ if and only if $p$ is maximal in $P$. 
	Then the fixed subspace spanned by the elements $v_p$ for $p$ maximal.
\end{lemma}

\begin{theorem}[Analogue of {\cite[Lemma 4.9]{MR1845897}}]
	The space of homogeneous conjugate invariant of degree $(N_1,N_2)$ for the complex Clifford group $\mathcal{X}_m$ of genus $m$ is spanned by $\ccwe(C(m))$, where $C$ ranges over all binary doubly-even self-dual codes of length $(N_1,N_2)$; This is a basis if $m + 1 \geq (N_1+N_2)/2$.
\end{theorem}

\begin{proof}
	Let $p$ be a parabolic invariant. 
	If $p$ is further a Clifford invariant, then
	\[ 
		X_P (h \otimes I_2 \otimes \cdots \otimes I_2) p = p.
	\]
	By \cref{lem:infinitesimally}, the operator acts triangularly on the vectors $\ccwe(C(m))$. 
	Since 
	\[
		\frac{2^{m-r} - 2^r}{2^m - 1} = 1 \iff r = 0,
	\]
	the hypotheses of \cref{lem:glossographical} are satisfied. 
	Now the theorem follows from \cref{lem:infinitesimally,lem:glossographical,thm:misshape}. 
	If $m + 1 \geq (N_1+N_2)/2$, then the complete conjugate weight enumerators are independent by \cref{thm:gluttonish}. 
\end{proof}

\begin{definition}
	Let $\rho : G \to GL_n(\mathbb{C})$ be an $n$-dimensional representation of a finite group $G$. 
	Then $G$ acts on the conjugate polynomial ring
	$$
		\mathbb{C}[x_1, \ldots, x_n, \bar{x}_1, \ldots, \bar{x}_n].
	$$
	Let
	$$
		\Inv(G) = \mathbb{C}[x_1, \ldots, x_n, \bar{x}_1, \ldots, \bar{x}_n]^G
	$$
	denote the ring of conjugate invariants, that is the ring of $G$-invariant conjugate polynomials. 
	Let $\alpha(N_1, N_2)$ be the dimension of the space of homogeneous conjugate polynomials of degree $(N_1,N_2)$. 
	Then Forger's theorem (generalization of Molien's series \cite{MR1600419,MR3159065}) states that
	\begin{equation}
		\sum_{N_1=0}^\infty \sum_{N_2=0}^\infty \alpha(N_1,N_2) t^{N_1} \bar{t}^{N_2} = \frac{1}{\abs{G}} \sum_{g \in G} \frac{1}{\det(I - t\rho(g))} \cdot \frac{1}{\det(I - \overline{t\rho(g)})}.
	\end{equation}
	We denote this series by $\FS_G(t,\bar{t})$.
\end{definition}

\begin{corollary}[Analogue of {\cite[Corollary 4.11]{MR1845897}}]
	Let $\FS_m(t,\bar{t})$ be the Forger series of the complex Clifford group of genus $m$. 
	As $m$ tends to infinity, the series converges monotonically as $m$ increases. 
	\[
		\lim_{m \to \infty} \FS_{\mathcal{C}_m(\rho)}(t,\bar{t}) = \sum_{N_1 = 0}^\infty \sum_{N_2 = 0}^\infty a_{N_1,N_2} t^{N_1} \bar{t}^{N_2},
	\]
	where $a_{N_1,N_2}$ is the number of permutation-equivalence classes of binary doubly-even self-dual codes of length $(N_1,N_2)$.
\end{corollary}

\begin{example}
	The inital terms in Forger series of $\mathcal{X}_1$ with degrees at most $8$ in $t$ as well as $\bar{t}$ are given by 
	\[
		1 + t\bar{t} + t^2\bar{t}^2 + t^3\bar{t}^3 + 2t^4\bar{t}^4 + 2t^5\bar{t}^5 + 3t^6\bar{t}^6 + 3t^7\bar{t}^7 + 4t^8\bar{t}^8 + t^8 + \bar{t}^8   
	\]
\end{example}

\begin{example}
	The inital terms in Forger series of $\mathcal{X}_2$ with degrees at most $8$ in $t$ as well as $\bar{t}$ are given by 
	\[
		1 + t\bar{t} + t^2\bar{t}^2 + t^3\bar{t}^3 + 2t^4\bar{t}^4 + 2t^5\bar{t}^5 + 4t^6\bar{t}^6 + 5t^7\bar{t}^7 + 8t^8\bar{t}^8 + t^8 + \bar{t}^8   
	\]
\end{example}

\begin{example} \label{empl:Erinaceus}
	By exhaustive search, we determine the following values of $a_{N_1,N_2}$, namely the number of permutation-equivalence classes of binary doubly-even self-dual codes of length $(N_1,N_2)$. 
	\[
		a_{1,1} = a_{2,2} = a_{3,3} = 1,\ a_{4,4} = a_{5,5} = 2,\ a_{6,6} = 4.
	\]
	We give the generator matrices $g_{1,1}, g_{4,4}, g_{6,6}^a, g_{6,6}^b$ for four indecomposable codes. 
	Note that one obtains a code of length $(N_1+N_1',N_2+N_2')$ by glueing together a code of length $(N_1,N_2)$ and a code of length $(N_1',N_2')$. 
	\[
	\begin{array}{cr}
			g_{1,1} =
		\begin{bmatrix}
			1 & | & 1
		\end{bmatrix} 
		&
		g_{6,6}^a = \left[
		\begin{array}{cccccc|cccccc}
		 1 & 0 & 0 & 0 & 0 & 0 & 1 & 1 & 1 & 1 & 1 & 0 \\
		 0 & 1 & 0 & 0 & 0 & 0 & 1 & 1 & 1 & 1 & 0 & 1 \\
		 0 & 0 & 1 & 0 & 0 & 0 & 1 & 1 & 1 & 0 & 1 & 1 \\
		 0 & 0 & 0 & 1 & 0 & 0 & 1 & 1 & 0 & 1 & 1 & 1 \\
		 0 & 0 & 0 & 0 & 1 & 0 & 1 & 0 & 1 & 1 & 1 & 1 \\
		 0 & 0 & 0 & 0 & 0 & 1 & 0 & 1 & 1 & 1 & 1 & 1 
		\end{array} \right] \\
		g_{4,4} = \left[
		\begin{array}{cccc|cccc}
			1 & 0 & 0 & 1 & 0 & 1 & 1 & 0 \\
			0 & 1 & 0 & 1 & 0 & 1 & 0 & 1 \\
			0 & 0 & 1 & 1 & 0 & 0 & 1 & 1 \\
			0 & 0 & 0 & 0 & 1 & 1 & 1 & 1
		\end{array} \right]
		&
		g_{6,6}^b = \left[
		\begin{array}{cccccc|cccccc}
		 1 & 0 & 0 & 0 & 1 & 0 & 0 & 0 & 0 & 0 & 1 & 1 \\
		 0 & 1 & 0 & 0 & 0 & 1 & 0 & 0 & 0 & 0 & 1 & 1 \\
		 0 & 0 & 1 & 0 & 1 & 1 & 0 & 0 & 1 & 1 & 1 & 0 \\
		 0 & 0 & 0 & 1 & 1 & 1 & 0 & 0 & 1 & 1 & 0 & 1 \\
		 0 & 0 & 0 & 0 & 0 & 0 & 1 & 0 & 1 & 0 & 1 & 1 \\
		 0 & 0 & 0 & 0 & 0 & 0 & 0 & 1 & 0 & 1 & 1 & 1 
		\end{array} \right]
	\end{array}
	\]
\end{example}

\begin{remark}
	The definition of projective $t$-design $X$ requires that the averaging of every function $f \in \Hom_{t,t}(\CC^d)$, homogeneous complex conjugate polynomials of degree $(t,t)$, over $X$ is equal to that over the complex unit sphere. 
	The number $a_{4,4} = a_{5,5} = 2$ tells us that there are exactly two linearly independent $\mathcal{X}_m$-invariant conjugate polynomial of degree $(4,4)$ or $(5,5)$. 
	The conjugate polynomial $f_1 = \sum_{i=1}^d z_i \overline{z_i}$, which is equal to constant $1$ when restricted to the complex unit sphere, is a $\mathcal{X}_m$-invariant conjugate polynomial of degree $(1,1)$. 
	Then $f_1^4$ is an invariant conjugate polynomial of degree $(4,4)$. 
	Suppose $f_4$ is the other invariant conjugate polynomial of degree $(4,4)$. 
	Consequently $f_1^5$ and $f_1 f_4$ are two invariant conjugate polynomial of degree $(5,5)$. 
	This implies that if an orbit of the complex Clifford group is a projective $4$-design, then it is automatically a projective $5$-design. 
\end{remark}

\section*{Acknowledgments}

The first author thanks TGMRC (Three Gorges Mathematical Research Center) in China Three Gorges University, in Yichang, Hubei, China, for supporting his visits there in April and August 2019 to work on the topics related to this research. 
The second author is supported by JSPS KAKENHI (17K05164).
The third author is supported in part by NSFC (11671258).

\bibliographystyle{plainurl}
\bibliography{complexClifford}

\end{document}